\documentclass[12pt]{iopart}

\usepackage{iopams}

\expandafter\let\csname equation*\endcsname\relax

\expandafter\let\csname endequation*\endcsname\relax

\usepackage{amsmath}

\usepackage{eucal}
\usepackage{a4wide}
\usepackage{amsfonts}
\usepackage{amsthm}
\usepackage{amssymb}
\usepackage{setspace}
\usepackage{xspace}
\usepackage{algorithm,algpseudocode}
\usepackage[T1]{fontenc}
\usepackage[utf8]{inputenc}
\usepackage{url}
\usepackage{graphicx}

\usepackage{algorithm,algpseudocode}
\usepackage{tikz}
\usepackage{array}
\usepackage{tabularx}
\usepackage[title]{appendix}
\usepackage[T1]{fontenc}
\usepackage{multirow}
\usepackage[utf8]{luainputenc}
\usepackage{graphicx}
\usepackage{subfig}
\usepackage{captdef}
\usepackage{float}
\usepackage{hyperref}
\usepackage{amsthm}
\usepackage{bbm}
\providecommand{\keywords}[1]{\textbf{Keywords:  } #1}

\newtheorem{thm}{Theorem}[section]
\newtheorem{prop}{Proposition}[section]
\newtheorem*{thmn}{Theorem}
\newtheorem{lem}{Lemma} [section]

\newtheorem{hip}{Hypothesis}[section]
\newtheorem*{pro}{Property}

\newcommand{\foune}{\mathcal{F}_{\rho \, , s}}

\newcommand{\klio}{k \,\mathcal{F}_{\rho \, , s} P(\xi)-1}
\newcommand{\klin}{k \,\mathcal{F}_{\rho \, , s} P(\xi)-1+(2\pi i \xi+j)\alpha}
\newcommand{\R}{\mathbb{R}}

\newcommand{\K}{\mathbb{K}}
\newcommand{\Ld}{L^2(\mathbb{R})}
\newcommand{\D}{\mathcal{D}(U)}
\newcommand{\PW}{PW_{\rho \, , s}(U)}
\newcommand{\HNL}{H_{\alpha}(N,\lambda)}
\newcommand{\HN}{H(N,\lambda)}
\newcommand{\Hg}{H^\alpha}

\newcommand{\Ne}{N_{\varepsilon}}
\newcommand{\lae}{\lambda_\varepsilon}
\newcommand{\HNLE}{H_{\alpha}(N_{\varepsilon},\lambda_{\varepsilon})}
\newcommand{\ldn}{L_{\rho \, , s}^2}
\newcommand{\eci}{k \mathcal{K}(H)-H}
\newcommand{\ecie}{k \mathcal{K}H_{\alpha}(N, \lambda)-H_{\alpha}(N, \lambda)}
\newcommand{\ecd}{\frac{d}{dx} gN+\lambda N}
\newcommand{\cosafea}{(2i \xi-s)\pi \, \mathcal{F}_{\rho \, , s}(gN \frac {d}{dx}\rho^{-1})(\xi)
\\
-\mathcal{F}_{\rho \, , s}(gN(\frac{d^2}{dx^2}\rho^{-1})(\frac{d}{dx}\rho^{-1})^{-1})(\xi)}

\newcommand{\derifeapr}{\frac {d}{dx}\rho^{-1}}
\newcommand{\derifeaseg}{( \frac{d^2}{dx^2}\rho^{-1})(\frac{d}{dx}\rho^{-1})^{-1}}

\newcommand{\SaN}{\alpha \frac{d}{dx} (\HNL (\derifeapr)^{-1})+ \alpha \HNL (\frac{d^2}{dx^2} \rho ^ {-1})(\frac{d}{dx} \rho ^ {-1})^{-2}+ j \alpha \HNL}

\begin{document}

%%%%%%%%-------- Functions

\title[Preprint]{Recovering the Fragmentation Rate in the Growth-Fragmentation Equation}

\author{Alvaro Almeida Gomez}

\address{Khalifa University, P.O. Box 127788, Abu Dhabi, United Arab Emirates}
\ead{alvaro.gomez@ku.ac.ae}

\author{Jorge P. Zubelli}

\address{Khalifa University, P.O. Box 127788, Abu Dhabi, United Arab Emirates}
\ead{zubelli@gmail.com}

\vspace{10pt}
\begin{indented}
\item[]February 2022
\end{indented}

\begin{abstract}
We consider the inverse problem of determining the fragmentation rate from noisy measurements in the growth-fragmentation equation. We use Fourier transform theory on locally compact groups to treat this problem for general fragmentation probabilities. We develop a regularization method based on spectral filtering, which allows us to deal with the inverse problem in weighted ${L}^2$ spaces. 
%Our approach regularizes the signal generated by differential operators in the frequency domain. 
As a result, we obtain a regularization method with error of order $O(\varepsilon^{\frac{2m}{2m+1}})$, where $\varepsilon$ is the noise level and $m>0$ is the {\em a priori} regularity order of the fragmentation rate.
\end{abstract}

% It is required to enter 2010 MSC.
% Please provide minimum  5 keywords.

\section{Introduction}
Growth-fragmentation equations describe, in a quantitative way, the evolution process of the density of an ensemble of particles that divide according to a certain fragmentation rate. In other words, we assume that each particle grows over time, and splits into smaller particles, in such a way that the total mass is conserved. This model is used in biological phenomena for instance, in cell division processes \cite{celdi1,celdi2,celdi3,laurent} and protein polymerization \cite{propol1}. It also appears in telecommunications \cite{teleco}. 
The goal of this article is to address the inverse problem for a class of such equations making use of the asymptotic behavior of their solutions which in turn, after proper scaling, converge to a stable distribution in time~\cite{celdi1,pery1,mimipe}.

The growth-fragmentation equation considered here can be expressed using the following integro-differential equation,
\begin{equation}
\left\{\begin{array}{lr}
\frac{\partial}{\partial t}n(t,x)+\frac{\partial}{\partial x}(g(x)n(t,x))=-B(x)n(t,x)+k \int_0^\infty  \K (x | y)\,B(y)\,n(t,y) dy,\\\\[0.05cm]
n(0,x)=n_{0}(x),\\ 
\\[0.05cm]
g(0)n(t,0)=0, \\ 
\end{array}\right.
\end{equation}
where $n(t,x)$ represents the density of particles of mass $x$ at  time $t$, with initial condition $n_{0}(x)$. The function $g(x)$ is the growth rate for particles of mass $x$. The function  $\K(x |y)$  represents the conditional probability   that a particle of mass $y$ splits into $k$ smaller particles of mass $x$. The function  $B(x)$ is the fragmentation rate of particles of mass $x$.
\par A natural question concerns to the asymptotic behaviour of the population density $n(t,x)$, when $t \rightarrow \infty$.  In \cite{pery1,mimipe}, it is shown that under certain conditions on the coefficients $g$, $B$, and $\K$, there exists an unique eigenpair $(N, \lambda)$ such that
\begin{equation}
\left\{\begin{array}{lr}
\frac{d}{dx}(g(x)N(x))+(B(x)+\lambda)N(x)=k \int_0^\infty \K(x | y)\,B(y)\,N(y) dy, \\\\[0.05cm]
g(0)N(0)=0,\\ 
\\[0.05cm]
N(x>0)>0 , \\ \\[0.05cm]
\int_{0}^{\infty} N(x)dx=1.
\end{array}\right.
\end{equation}
Moreover, in appropriate weighted norms, we have that $n(t,x)e^{-\lambda t} \rightarrow N(x),$ when the time $t$ goes to infinity.

Thus, we shall focus on the pair $(N(x), \lambda)$,
which corresponds to the eigenpair of a linear operator \cite{pery1}. They have the biophysical interpretation of the asymptotic distribution of the population density and the corresponding convergence rate. In principle, such quantities could be experimentally observed in some situations due to the exponential decay rate.  This was used for instance in~\cite{domazu}.
The function $N$ alluded to above is the stable distribution which is observed in a broad class of structured population models \cite{celdi1}.
The importance of studying $ N (x) $ instead of $ n (t, x) $ is that the variable $ t $ is removed, which reduces the study of a two-dimensional problem to a one-dimensional. 

\par We are interested in how to recover in a stable way the fragmentation rate $B(x)$ from noisy measurements of $N(x)$ and $\lambda$. Our strategy is to study $H=BN$ rather than $B$, as proposed in \cite{dopezu,bodoes}, and then use truncated division by $N$ to recover $B$. Thus, the inverse problem under consideration is how to recover in a stable way $H$ from approximate knowledge of $N$ and $\lambda$, in which $H$ is the solution of 
\begin{equation} 
\label{meq}\eci =\ecd ,
\end{equation}
where 
\begin{equation}
    \mathcal{K}(H)=\int_0^{\infty}\K(x | y) \, H(y)\,dy .
    \label{operadork}
\end{equation}
In \cite{dopezu,pezu,domazu} the inverse problem was treated for the equal mitosis case, that is, when a particle splits into two small particles with half of its mass. In this case, the parameters are
\begin{equation}
   k=2 \quad \textrm{and} \quad\K(x | y)=\delta_{x=\frac{y}{2}} 
\end{equation} 
A more general class of conditional probability density kernels is given by the self-similar ones, that is, when 
\begin{equation}
    \K(x|y)=\frac{1}{y}P(\frac{x}{y}),
\end{equation}
where $P$ is a probability measure in $[0,1]$. Self-similar probabilities arise when the fragmentation depends on the ratio between the size of the mother particle and the size of the next-generation particle. The inverse problem for self-similar probabilities was treated in the works \cite{bodoes,doti}. \par The aim of the present work is to address this inverse problem for more general probabilities. Namely, we consider conditional probabilities of the form
\begin{equation}\label{ref:fragmenker}
    \K(x|y)= P(x\odot_{\rho} y^{-1}) \,\, \, \frac {d}{dz}\rho^{-1}(y), 
\end{equation}
where $\rho$ is an increasing diffeomorphism from $\R$ to $\mathbb{R^{+}}$, and $P$ is a probability on $(0,\rho (0))$. The transport operation $\odot_{\rho}$ is defined by
\begin{equation}
    a \odot_{\rho} b := \rho ( \,\rho^{-1} (a)+ \rho^{-1} (b)\,).
\end{equation}
 We shall denote conditional densities in \eqref{ref:fragmenker}  as transport probabilities. Observe that if we take the exponential function $\rho(x)=e^x$, then the transport operation for this function is the usual product on the positive real line.  Hence, transport probabilities generalize self-similar probabilities. 
\par We treat the inverse problem described in Eq.~\eqref{meq}. Here, we do not require that the mass must be conserved in the division process. This is in order not to restrict this method only to biological models, and also to use it in real-world applications. treating this problem from a more general point of view. We do that in two steps. Firstly, we guarantee that under some assumptions in $N$, there exists a unique $H$ solution of Eq.~\eqref{meq}. To do that, we establish that on proper spaces the operator $k \mathcal{K}-Id$ is an isomorphism. The second step is to guarantee the stability of the inverse problem. For that, we propose a new regularization method. This method is based on the implementation of the Fourier transform theory, and spectral filtering techniques.
\par We show that the Quasi-Reversibility approach \cite{pezu} is a particular case of our method.  Compared to regularization methods based on convolution \cite{bodoes,doti}, our method improves the error order. Namely, we obtain an error of order $O(\varepsilon^{\frac{2m}{2m+1}})$, where $\varepsilon$ is the noise level.
\par This article is organized as follows: in Section \ref{preli}, we study the transport operation and some relations with the Fourier transform on locally compact abelian groups. In Section \ref{sectinver}, we discuss the invertibility of the operator $k \mathcal{K}-Id$ in proper spaces. In Section \ref{secregu}, we present a new regularization method to treat the stability of the inverse problem. In Section \ref{ejempl}, we give examples for which some of the Hypotheses Eq.~\eqref{hipbobe} or Eq.~\eqref{hipunbobe} are satisfied. Finally, in Section \ref{numer}, we present the numerical implementation of our method.

\section{Preliminaries}\label{preli}
To fix the notation, we briefly review some concepts of Fourier transform for locally compact groups. We refer the reader to \cite{rud} for more details. We recall that the Haar measure on an abelian group $G$ is the unique non-negative and regular measure, up to a positive multiplicative constant, which is translation invariant. Here, by translation, we mean the action of multiplication by group elements. 
\par We define the transport operation $\odot_{\rho}$ as
\begin{equation}
    a \odot_{\rho} b := \rho ( \,\rho^{-1} (a)+ \rho^{-1} (b)\,),
\end{equation}
where $\rho$ is an increasing diffeomorphism from $\R$ to $\mathbb{R^{+}}$.
The pair $(\mathbb{R^{+}},\odot_{\rho}),$ equipped with the induced topology of $\mathbb{R},$ is a locally compact group. Moreover, the Haar measure $\mu_{\rho}$ is given by the  measure 
\begin{equation}
  d\mu_{\rho}=\frac {d}{dx}\rho^{-1} dx ,  
\end{equation}
where $dx$ is the Lebesgue measure on the positive real line. In this context, we can develop the theory of the Fourier transform on the group $(\mathbb{R^{+}},\odot_{\rho},\mu_{\rho})$.

\subsection{Fourier transform on \texorpdfstring{$(\mathbb{R^{+}},\odot_{\rho},\mu_{\rho})$.}{fouriertrasnform}}
The Fourier transform $\mathcal{F_{\rho}}$ on the group $(\mathbb{R^{+}},\odot_{\rho})$ is defined for a function  $ f(x)\in L^1(\mathbb{R^{+}},\mu_{\rho}dx)$ and a real number $\xi$ as
\begin{equation}
    \mathcal{F_{\rho}}f(\xi)=\int_0^\infty  f(x) \, e^{-2 \pi i \xi \rho^{-1}(x)}\, d \mu_{\rho}.
\end{equation}
In this general context, the Fourier transform theory on  $(\mathbb{R^{+}},\odot_{\rho})$ can be developed in a similar fashion as in the standard case $ (\mathbb{R},+)$. In particular, we obtain the inversion and the Plancherel theorems. See \cite{rud} for more details.
\begin{thmn}[Inversion theorem~\cite{rud}]
Suppose that $ f\in L^1(\mathbb{R^{+}},\mu_{\rho}dx)$ and $\mathcal{F_{\rho}}f\in L^1(\mathbb{R},dx),$ then for a.e positive number $x$ we have
\begin{equation}
  f(x)=\int_{-\infty}^{+\infty} \mathcal{F_{\rho}}f(\xi) \,e^{2 \pi x \rho^{-1}(\xi)}\, d\xi.  
\end{equation}
\end{thmn}
\begin{thmn}[Plancherel Theorem~\cite{rud}]
The Fourier transform $\mathcal{F_{\rho}}$, restricted to $L^1(\mathbb{R^{+}},\mu_{\rho}dx) \cap \, L^2(\mathbb{R^{+}},\mu_{\rho}dx)$ is an isometry $( with\,respect\, to \,the\, L^2 \,norm)$ onto a dense linear subspace of $L^2(\mathbb{R},dx)$. Hence, it may be extended, in a unique manner, to an isometry from $L^2(\mathbb{R^{+}},\mu_{\rho}dx)$ to $L^2(\mathbb{R},dx).$
\end{thmn}

\subsection{\texorpdfstring{The $L_{\rho \, , s}^2$ spaces.}{L2spac}}

In order to extend the Plancherel theorem, we define the $L_{\rho \, , s}^2$ space. They are a natural generalization of the $L^2(\mathbb{R^{+}},\mu_{\rho}dx)$ space. More precisely, the $L_{\rho \, , s}^2$ space is defined as $L^2(\mathbb{R^{+}},e^{2\pi  s \rho^{-1}(\cdot)}\mu_{\rho}).$
Observe that there exists a canonical isomorphism $T_{s}$ between  $L_{\rho \, , s}^2$ and  $L_{\rho \, , 0}^2=L^2(\mathbb{R^{+}},\mu_{\rho}dx),$ which is given by
\begin{equation}
 T_{s}f(x)=f(x) \,e^{\pi  s \rho^{-1}(x)}.   
\end{equation}
We now define the Fourier transform $\mathcal{F}_{\rho \, , s}$ on $L_{\rho \, , s}^2$ as
\begin{equation}
  \mathcal{F}_{\rho \, , s}=\mathcal{F}_\rho \, \circ \,T_{s}.  
\end{equation}
Thus, by the Plancherel Theorem, the generalized Fourier transform $\mathcal{F}_{\rho \, , s}$ is an isometry.

%\subsection{The Fourier transform of probability measures}
For the case when $P$ is a probability measure on $R^{+}$ satisfying
\begin{equation}            
\label{conditionker} \int_0^\infty e^{\pi s \rho^{-1}(z)}\frac {d}{dx}\rho^{-1}(z) dP(z)<\infty \mbox{, }
\end{equation}
we define the Fourier transform $\mathcal{F}_{\rho \, , s}$ of the probability $P$ by
\begin{equation}
  \mathcal{F}_{\rho \, , s}P (\xi)= \int_0^\infty e^{(-2 \xi i+s)\pi   \rho^{-1}(z)}\frac {d}{dz}\rho^{-1}(z) dP(z),  
\end{equation}
for any real number $\xi.$
\subsection{The Fourier transform and the convolution operator}
We now consider the convolution-type operator induced by the probability $P$ and  given by
\begin{equation}  
\label{opker} \Tilde{\mathcal{K}}(f)(x)=\int_0^\infty P(x \odot_{\rho} y^{-1} )f(y)\, d\mu_{\rho}(y).
\end{equation}
In the same manner, as in the case of $\mathbb{R},$ and under some conditions of integrability, the Fourier transform $\mathcal{F}_{\rho \, , s}$ of a convolution operator becomes a multiplicative operator.
\begin{pro}Let $P$ be a probability measure satisfying Equation  \eqref{conditionker}, then the convolution operator defined in Eq. \eqref{opker} satisfies the multiplication property
\begin{equation}
\label{convmul}\mathcal{F}_{\rho \, , s} \Tilde{\mathcal{K}}(f)(\xi)=\mathcal{F}_{\rho \, , s}P (\xi)  \, \cdot \, \mathcal{F}_{\rho \, , s}f(\xi),
\end{equation}
for all $f$ in $L_{\rho \, , s}^2$.
\end {pro}

%%-------------Fourier transform and differentiation------------
\subsection{The Fourier transform of \texorpdfstring{$\frac{d}{dx} gN$.}{fourierderi} }

The Fourier transform on the real line has the property of diagonalizing differential operators. We now state an analogous version of this fact for the Fourier transform $\mathcal{F}_{\rho \, , s}$.
%-------proposition 
\begin{prop}
\label{fouanddif} Suppose that $S$ belongs to the Sobolev space $H^{1}_0 (0, \infty)$,  and that the functions
\begin{equation}
 \label{condiderivat}\frac{d}{dx} S\,  , \, \, S \frac {d}{dx}\rho^{-1} \,  , \, \, S(\frac{d^2}{dx^2}\rho^{-1})(\frac{d}{dx}\rho^{-1})^{-1},
\end{equation}
are in $L_{\rho \, , s}^2$. Then, we have

\begin{align}
\label{quasemultipl}
\begin{split}
\mathcal{F}_{\rho \, , s}(\frac{d}{dx} S)(\xi)= (2 \xi i-s)\pi \, \mathcal{F}_{\rho \, , s}(S \frac {d}{dx}\rho^{-1})(\xi)
\\
-\mathcal{F}_{\rho \, , s}(S(\frac{d^2}{dx^2}\rho^{-1})(\frac{d}{dx}\rho^{-1})^{-1})(\xi) .
\end{split}
\end{align}
\end{prop}

\begin{proof}
Using integration  by parts, and  the smoothness assumption of the function $S$,  we have that
$$\mathcal{F}_{\rho \, , s}(\frac{d}{dx} S)(\xi)=-\int_0^\infty  S(z) \frac{d}{dz} \, (e^{(-2 \xi i+s) \pi \rho^{-1}(z)}\, \, \frac{d}{dz}\rho^{-1}(z)) \, \, dz.$$
Now, we observe that 
$$  \frac{d}{dz} \, (e^{(-2 \xi i+s) \pi \rho^{-1}(z)}\, \, \frac{d}{dz}\rho^{-1}(z)) = ((-2 \xi i+s)\pi (\frac{d}{dz}\rho^{-1}(z))^2+ \frac{d^2}{dz^2}\rho^{-1})e^{(-2 \xi i+s) \pi \rho^{-1}(z)},$$
 and this proves our result.
\end{proof}

%%----------------end Fourier transform and differentiation-----------

\section{Invertibility of the operator \texorpdfstring{ $k \mathcal{K}-Id$.}{Invofoperator}} \label{sectinver}
 Let $\mathcal{K}: L_{\rho \, , s}^2 \to L_{\rho \, , s}^2 $ be  the operator defined as 
 \begin{equation}
    \mathcal{K}(H)=\int_0^{\infty}\K(x,y) \, H(y)\,dy .
\end{equation} 
We  use the ideas  of \cite{bodoes} in order to guarantee the invertibility of $k \mathcal{K}-Id$. We  prove that under additional assumptions on the probability $P$, the operator  $k \mathcal{K}-Id$, has a bounded inverse in some subspaces of $L_{\rho \, , s}^2$. If we apply the Fourier transform and the Eq.~\eqref{convmul}, we get that $k \mathcal{K}-Id$ is in fact a multiplication operator 
\begin{equation}
    \mathcal{F}_{\rho \, , s} (k \mathcal{K}-Id)f(\xi)=(k \,\mathcal{F}_{\rho \, , s} P(\xi)-1) \,\mathcal{F}_{\rho \, , s} f (\xi) .
\end{equation}
Thus, the operator $k \mathcal{K}-Id$ has a bounded inverse on $L_{\rho \, , s}^2$ if the function $$|k \,\mathcal{F}_{\rho \, , s} P(\cdot)-1 |,$$ is bounded from below by positive constant on the real line. 
In some cases, the function $|k \,\mathcal{F}_{\rho \, , s} P(\cdot)-1 |$ never vanishes, but goes to zero when $x$ converges to $0$ or $\infty$, in this situation, the function $|k \,\mathcal{F}_{\rho \, , s} P(\cdot)-1 |$ is bounded from below on compact sets of the real line. Moreover, for computational purposes, we are interested in a reconstruction on compact intervals. In this case, we focus on the invertibility of the operator on compact sets. To find a bounded inverse on an open set $U$, we assume that the probability $P$ satisfies the following hypothesis.
\begin{hip}
\label{hipbobe}There exists an open set $U,$ such that the function $|k \,\mathcal{F}_{\rho \, , s} P(\cdot)-1 |$ is bounded from below by a positive constant on $U$.
\end {hip}
In Section \ref{ejempl} we give examples where the previous hypothesis is satisfied. 
\par  We consider the Paley-Wiener spaces $\PW$, as the subspace of $L_{\rho \, , s}^2$ of all the functions $f$  whose Fourier transform $\mathcal{F}_{\rho \, , s}$ has support on $U$.
Observe that if $f$ is on $\PW$, then $\mathcal{K} f$ also lies on $\PW$. Thus, the operator $k \mathcal{K}-Id$ from $PW_{\rho \, , s}^{\alpha}$ to itself is well defined. Using the Fourier transform $\foune$, and Eq. \eqref{convmul}, we conclude the following proposition, which guarantees the existence of a bounded inverse.

\begin{prop}
Suppose that the probability $P$ satisfies Hypothesis \ref{hipbobe}, then the operator $k \mathcal{K}-Id$ from $\PW $ to itself has a bounded inverse.
\end {prop}

\section{Regularization of the inverse problem} \label{secregu}

In this section, we discuss some methods to regularize the inverse problem, i.e, recover $H=BN$ from noisy measurements of $N$ in some $\ldn$ norm, in such a way that we can control the approximation error.
\par The main difficulty arises from the fact that we cannot estimate the $\ldn$ norm of $\frac{d}{dx} gN$ from $\|N\|$. The principal idea in this section is to use the Fourier transform to turn the differential operator $\frac{d}{dx}$ into a  Fourier multiplication operator \cite{xavierdua}, and then, we regularize the inverse problem using spectral filters as described in \cite{reguinvlib}.

\subsection{Regularization by spectral filtering}
We now present our strategy to regularize the inverse problem as follows. Using the Proposition \ref{fouanddif}, we see that under the Fourier transform, the differential operator is simply a quasi-multiplication operator, whose multiplication part is given by   $(2i \xi-s) \pi$. Unfortunately, this function is unbounded in non-compact sets. Then, noisy data with small errors can approximate solutions with large errors, that is, it has a de-regularizing effect. To  regularize the inverse problem, we define a multiplication approximation for $\frac{d}{dx}gN$, for which the multiplier is given by $(2i \xi-s) \pi$ by a regularized filter $f(\alpha,\xi)$. To be more precise, we define the approximation $\HNLE$ by 
\begin{equation}
\label{apro} \foune \HNL (\xi)=f(\alpha,\xi)\,L(N)(\xi)+ \, h(\alpha,\xi) \, \lambda \, \foune N(\xi),
\end{equation}
where
$$L(N)(\xi)=\cosafea.$$
Let $(\Ne,\lae)$ be a noisy measurement of $(N,\lambda)$ in the product space $\PW\, \times \,L^{\infty}(0,\infty) $, and let $H=\HN$ be the unique solution of Eq.~\eqref{meq} in $\PW$. The following theorem gives error estimates for $\|\HNLE-H(N,\lambda)\|$. It was inspired by \cite{dopezu}, yet, the method is different because it focuses on regularizing the spectral signal of the differential operators.

\begin{thm}[Spectral regularization]\label{spectregu} Let K be a subspace of $\PW$. Suppose that there exists a positive constant $C$, such that for all $\alpha \, \in [0,C)$ the bilinear operator $H_\alpha$ defined from $K \, \times \,L^{\infty}(0,\infty) $ to $\PW$ is bounded. Assume that $N \,\in \ldn$ satisfies
$$\|\HN-\HNL \| \le \alpha M,$$
where the constant $M$ does not depend on $\alpha$. Then, we have the estimate
$$ \|\HNLE-\HN\|\le \|H_{\alpha} (\Ne-N,\lae-\lambda)\|+\alpha M. $$
\end{thm}
\begin{proof}
By the triangle inequality, we obtain
\begin{equation*}
\|\HNLE-\HN\|\le\|\HNLE-\HNL\|+\|\HNL-\HN\|,
\end{equation*}
which implies the result.
\end{proof}
To apply the above result, we need to guarantee that for a fixed $\alpha$, the operator $H_{\alpha}$ is well defined and bounded. For simplicity, we first consider the case when 
\begin{equation}
\label{condicre}  g\derifeapr, \quad g\derifeaseg, 
\end{equation}
and
\begin{equation}
\label{condifil}
\xi \to \quad \xi f(\alpha , \xi), \quad f(\alpha , \xi), \quad h(\alpha,\xi), 
\end{equation}
are bounded functions on $U$.
To show that the operator $H_\alpha$ is well defined, we consider the subspace $\D$ of $\PW$, which consists of all functions $N$ such that
$$gN\derifeapr, \quad gN\derifeaseg, $$
are in $\PW$. In that case, it is straightforward to show that the operator $H_{\alpha}$ from $\D \, \times \,L^{\infty}(0,\infty) $ to $\PW$ is well defined, and bounded. We use modified versions of Tikhonov and Landweber filters,  to show that all the functions in Equation \eqref{condifil} are bounded. These filters are commonly  used in regularization theory for compact  operators \cite{reguinvlib}.

%%------------End spectral filtering

\subsection{Tikhonov filtering}
In the classical case of linear operators, the solution $x$ to the problem $Ax=y$, where $A: \mathcal{H} \to  \mathcal{H}$ is a non-negative linear and compact operator, can be regularized by using the filter
$$\lambda \mapsto \frac{1}{\lambda+\alpha},$$
in the spectral variable.  Indeed, considering the singular value decomposition of $A$ given by 
$$A(x)= \sum_{n=1}^{\infty}  {\lambda_{n}} <u_n, x> v_n \mbox{ ,} $$
the regularized solution $x^\alpha$ to the problem
$A x = y  $ is given by 
$$x^{\alpha}=\sum_{n=1}^{\infty}  \frac{<v_n,y>}{\lambda_{n}+\alpha} u_n \mbox{ .}$$
See \cite{reguinvlib} for more details. For the problem under consideration, we modified the above function and consider the filter
\begin{equation}\label{tikfil}
f(\alpha,\xi)=  \frac{1}{\klio}\left( \frac{1}{1+\alpha  |\xi|} \right) \quad  \textrm{and} \quad h(\alpha,\xi)=\frac{1}{\klio} . 
\end{equation}
We assume that  the probability $P$ satisfies Hypothesis \ref{hipbobe}. Using these filters, we observe that the functions in Equation  \eqref{condifil} are bounded with respect to the variable $\xi$. Under these assumptions, the bilinear operator $H_{\alpha}$ is bounded. In fact, a straightforward computation shows that
$$\|H_{\alpha} (N, \lambda)\| \le C ( 1+\frac{1}{\alpha}+|\lambda|) \|N\|,$$
for some positive constant $C$, which only depends on $s$, $P$, $g$, $\rho ^{-1}$. 
\par We now state the following regularization method, based on the filter functions defined above. 

%----- Tikhonov regularization method----------

\begin{thm}[Tikhonov regularization] \label{tikteo}Assume that the probability $P$ satisfies Condition \ref{conditionker}, and  Hypothesis \ref{hipbobe}. Moreover, we assume that $N \in \D$ satisfies all the assumptions of Proposition \ref{fouanddif}, and 
$$\xi \, \foune(\frac{d}{dx} gN)  \in \Ld \, .$$
Then, for a noisy measurement $(\Ne,\lae)$ of $(N,\lambda)$ in $\D \, \times \,L^{\infty}(0,\infty) ,$ the approximation $\HNLE$ using the filters Eq.~\eqref{tikfil} satisfies the  estimate 
$$ \|H-\HNLE\| \le \,M\,(\alpha+(1+\frac{1}{\alpha}+|\lae-\lambda|) \,\|\Ne-N\|), $$
for some positive constant $M$ which depends on $s$, $P$, $g$, $\rho ^{-1}$, and $N$. 
\newline If $(\Ne,\lae)$ satisfy $\|\Ne-N \| \le \varepsilon$ , and 
$\|\lae-\lambda\| \le \varepsilon$. Then, we can choose the optimal parameter $ \alpha=\sqrt{\varepsilon} $ to conclude 
$$ \|H-\HNLE\| \le M(\sqrt{\varepsilon}+\varepsilon+\varepsilon^2). $$
\end{thm}

\begin{proof}
Using the Fourier transform in Eq. \eqref{meq}, together with Proposition \ref{fouanddif}, we have
$$ \foune(\HN-\HNL)= \frac{1}{\klio} \, \frac{\alpha |\xi|}{1+\alpha |\xi|}  \, \foune (\frac{d}{dx} gN ).$$
Therefore, we obtain
$$\| H-\HNL\| \le \alpha M' \left\|\xi \, \foune \,\left( \,\frac{d}{dx} gN \,\right) \right\|   \mbox{ ,}$$
for some constant $ M'$.
Hence, by Theorem \ref{spectregu} we obtain the  result.
\end{proof}

%--------------end tikhonov----------------

%%---------Quasi-reversible Tikhonov--------

\subsection{The quasi-reversibility Tikhonov filtering}
The quasi-reversibility method proposed in  \cite{pezu} regularizes the inverse problem for the case of equal-mitosis. The idea of quasi-reversibility goes back to the work  \cite{lion}. This method is based on adding a perturbation of a differential operator. In \cite{doti}, an extension of this method was proposed for more general probabilities. We now present a different generalization using spectral filters. In fact, we show that there exists a relation between the quasi-reversibility method and our approach in the case of self-similar probabilities. 
\par Let us first describe our method without going into technical details. For each real number $j$, we defined the following modified Tikhonov filters
$$f^{j}(\alpha,\xi)=h^{j} (\alpha,\xi)=\frac{1}{\klin}.$$
If the quotient factor blows up, then the last equation is not well defined. To avoid this kind of problem, we assume that the probability $P$ satisfies a stronger hypothesis than \ref{hipbobe}, namely, we assume that
\begin{hip}
\label{hipunbobe}	There exist positive constants $M$ and $C$ such that $|\klin|\ge M$ for all $\xi \, \in U,$ and $\alpha \, \in [0,C]$.
\end {hip}
In Section \ref{ejempl}, we give some examples for which the above hypothesis is satisfied. If we assume that $N$ is  a smooth function  satisfying Condition \ref{suavidadcondi}, we see that under the Fourier transform $\foune$, the approximation in Eq.~\eqref{apro} solves the following differential equation
\begin{equation} \label{quasereg}
S_{\alpha}(N,\lambda)+\ecie=\ecd,
\end{equation}
where
$$S_{\alpha}(N,\lambda)= \SaN.$$
Thus, our method can be seen as a perturbation method~\cite{kato1995perturbation}. For the self-similar case, that is, when $\rho (x)=e^x$, we see that
$$S_{\alpha}(N,\lambda)=\alpha \frac{d}{dx} (x \HNL)+\alpha (j-1)\HNL .$$
Taking $j=1$,  the perturbation method defined in Eq. \eqref{quasereg}  reduces to the quasi-reversibility method proposed in \cite{doti}. Using the same ideas of the proof of Theorem \ref{tikteo}, we  prove the next result.

\begin{thm}[Quasi-reversibility method]

Assume that the probability $P$ satisfies Eq.~\eqref{conditionker} and Hypothesis~\ref{hipunbobe}. Moreover, assume that $N \in \D$ satisfies all the assumptions of Proposition \ref{fouanddif}. If the functions
$$\xi \, \foune(\frac{d}{dx} gN)  \quad and \quad \xi \, \foune(\frac{d}{dx} N), $$
are in $ L^2(\R)$, then, for all noisy measurement $(\Ne,\lae)$ of $(N,\lambda)$ in $\D \, \times \,L^{\infty}(0,\infty),$ the approximation $\HNLE$ satisfies the error estimate 
$$ \|H-\HNLE\| \le \,M\,(\alpha +(1+\frac{1}{\alpha}+ \|\lae-\lambda\|) \,\|\Ne-N\|, $$
for some constant $M$ which depends on  $P$, $g$, $\rho ^{-1}$, and $N$. 
\newline If $(\Ne,\lae)$ satisfies $\|(\Ne-N,\lae-\lambda)\| \le \varepsilon,$ we can choose the regularization parameter $ \alpha=\sqrt{\varepsilon} $, to conclude 
$$ \|H-\HNLE\| \le M(\sqrt{\varepsilon}+\varepsilon+\varepsilon^2). $$

\end{thm}

%%---------Landweber filter.--------

\subsection{The Landweber‘s method}
Using the Tikhonov and quasi-reversibility filters, we obtain an approximation error of order $O(\, \sqrt{\varepsilon} )$. We improve this error order, and thus we  obtain a better approximation. For that, we use a new filter, which is a modification of the classical Landweber filter
\begin{equation}
    \label{landweberfilterregulari}
    f(\alpha,\xi)=\frac{1}{\klio} \left(1-\left(1-\frac{1}{1+\xi^2} \right)^{{\alpha}}\right) \quad
\textrm{and} \quad h(\alpha,\xi)=\frac{1}{\klio} .
\end{equation}
To implement these filters in Theorem \ref{spectregu}, we first establish the following inequalities.

%%-----------------3 lemma--------

\begin{lem}
\label{tt} 
There exists a positive constant $A$, such that for all $x \in \R_{>0}$ and for all positive $\alpha$, the following estimate holds
$$| xf(x,\alpha)| \le A\sqrt{\alpha} ,$$
where $A$ depends on $s$, $P$, $g$, $\rho ^{-1}$. Moreover, for all $m \in \R$, such that  $m<2\alpha$, the following inequality holds
$$| \frac{1}{x^m} (1-\frac{1}{1+x^2})^{\alpha}| \le (\frac{m}{2 \alpha})^m .$$

\end{lem}

\begin{proof}

Let us prove the first inequality. Using the change of variables $z=1-\frac{1}{1+x^2}$, we see that
$$| xf(x,\alpha)|^2 \le M \frac{z}{1-z} (1-z^{\alpha})^2.$$
Since $z \in [0,1],$  then applying  the mean value inequality to the function $z^{\alpha}$, we obtain
$$ (1-z^{\alpha})^2 \le (1-z^{\alpha}) \le \alpha (1-z),$$
thus $| xf(x,\alpha)| \le M \sqrt{\alpha}$ as desired.
To prove the second inequality, we observe that at the point $ \omega=\frac{2\alpha-m}{2\alpha} $, the function $(1-z)^m z^{2\alpha-m}$  attains its  maximum for $x \in [0,1]$. Thus, for all $x \in \R$ we have
$$ | \frac{1}{x^m} (1-\frac{1}{1+x^2})^{\alpha}|^2 =(1-z)^{m} z^{2 \alpha-m} \le (1-\omega)^m=(\frac{m}{2\alpha})^m.$$
 \end{proof}

%%------------------end lemma----------------------------------
As a consequence of the previous lemma, we have  that the functions in  Eq.~\eqref{condifil} are bounded. Thus, the bilinear operator $H_{\alpha}$ is bounded. In fact, a straightforward computation shows that
$$ \|H_{\alpha}(N,\lambda)\|\le A( 1+{\sqrt{\alpha}}+ \|\lambda\|) \|N\|,$$
for some positive constant $A$, which only depends on $s$, $P$, $g$, $\rho ^{-1}$. Now, we apply Theorem~\ref{spectregu} to the Landweber filter\cite{reguinvlib}. For that, we require a smoothness condition of order $m>0$ for the function $N$
%-------------------------------------Landweber-----------------

\begin{thm}[Landweber regularization] \label{landw} Assume that the probability $P$ satisfies Eq.~\eqref{conditionker}, and Hypothesis \ref{hipbobe}. Moreover, assume that $N \in \D$, satisfies all the assumptions of Proposition \ref{fouanddif}, and the smoothness condition with order $m \in \R_{>0}$
$$\xi^{ m+1} \, \foune(\frac{d}{dx} gN)  \in L^2(\R) \, .$$ 
Then, for all noisy measurement $(\Ne,\lae)$ of $(N,\lambda)$ in $\D \, \times \,L^{\infty}(0,\infty) $, the approximation $\HNLE$ defined in Eq.~\eqref{apro} using the spectral filter of Eq.~\eqref{landweberfilterregulari}  satisfies the error estimate 
$$ \|H-\HNLE\| \le \,M\,((\frac{m}{\alpha})^m +(1+\sqrt{\alpha}+\|\lae-\lambda)\|) \,\|\Ne-N\|), $$
for some constant $M$ which depends on $s$, $m$, $P$, $g$, $\rho ^{-1}$, and $N$. Suppose that $(\Ne,\lae)$ satisfies $\|(\Ne-N,\lae-\lambda)\| \le \varepsilon <2^{\frac{2m+3}{2}} \sqrt{m} $. We can choose the optimal parameter given by
$$ \alpha=(\frac{2m^{m+1}}{\varepsilon })^\frac{2}{2m+1}, $$
which satisfies $m<2\alpha$, to conclude 
$$ \|H-\HNLE\| \le M(\varepsilon+\varepsilon^\frac{2m}{2m+1}). $$

\end{thm}

\begin{proof}
Using the Fourier transform in Equation \eqref{meq}, and Proposition \ref{fouanddif}, we have that
$$ \foune(\HN-\HNL)= \frac{1}{\klio} (1-\frac{1}{1+x^2})^{\alpha} \foune (\frac{d}{dx} gN ).$$
Next, we apply the second estimate of Lemma \ref{tt} to obtain
$$ \| \HN-\HNL\| \le M(\frac{m}{\alpha})^m \| \xi^{ m+1} \, \foune(\frac{d}{dx} gN) \|,$$
hence by Theorem \ref{spectregu}, we obtain the result. 
\end{proof}

%%-----General case---- Unbounded functions-------------
\subsection{The unbounded case}

In many cases, the  functions in Equation \eqref{condicre} are not bounded. For instance, in the self-similar case. We now  extend our results to the unbounded case.  The idea is to regularize the functions in Equation \eqref{condicre} using a spectral filter. To do that, we write
$$T_1= \left(\derifeapr\right)^2 \quad  T_2= \left(\derifeaseg \right)^2,$$
and
$$g_{\alpha}=\frac{ g  }{1+ \alpha \,(\exp(g^2)+\exp(g^2 T_1)+\exp(g^2 T_2))},$$
where $\alpha \in \R_{>0},$ is the regularization parameter of the functions \eqref{condicre}.
Observe that the exponential decay guarantees that using $g_\alpha$  in the place of $g$, then the functions in Eq.~\eqref{condicre} are bounded. Now, we  show that if the function $N$ has fast decay, then the bounded function $g _\alpha$ can be used to regularize the inverse problem, even if $g$ is not bounded.
\begin{prop}Assume that the function $N$ satisfies
\label{propositionregulanolimi}
\begin{equation}
    \label{suavidadcondi}
    \frac{d}{dx} (g_\alpha (\exp(g^2)+\exp(g^2 T_1)+\exp(g^2 T_2) ) N)\, \, \in \ldn.
\end{equation}
We define $\Hg$ as the solution of Eq.~\eqref{meq}, where the solution is associated with the  function  $g_{\alpha}$ instead of $g$. Then, we have that
$$\|H-\Hg\| \le C \alpha\| \frac{d}{dx} (g(\exp(g^2)+\exp(g^2 T_1)+\exp(g^2 T_2) )N)\|,$$
where $C$ is a constant, which depends on the operator $\mathcal{K}$ and the number $k$.
\end{prop}

\begin{proof}Observe that
$$ (k \mathcal{K}-Id)(H-\Hg)= {\frac{d}{dx} (g-g_\alpha)N}.$$
Since
$$g-g_\alpha=g \left( \frac{\alpha \exp(g^2)+\exp(g^2 T_1)+\exp(g^2 T_2)}{1+ \alpha \,(\exp(g^2)+\exp(g^2 T_1)+\exp(g^2 T_2))} \right),$$
then, using the Fourier transform $\mathcal{F}_{\rho \, , s}$, we obtain
$$ \left\| H-\Hg \right\| \le \alpha \left\|\frac{1}{\mathcal{F}_{\rho \, , s} (k \mathcal{K}-Id)}\right\|_{\infty} \left \| \frac{d}{dx} (g(exp(g^2)+exp(g^2 T_1)+exp(g^2 T_2) )N) \right\|$$
\end{proof}
Thus, the function $\Hg$ is a controlled approximation for $H$. In this case, the function $\Hg$ is the solution of Eq.~\eqref{meq}, associated with the bounded function $g_\alpha$. Therefore, we can use some of the previous methods (Tikhonov or Landweber) to regularize  $\Hg$, and by Proposition \ref{propositionregulanolimi} to regularize $H$.

\section{Examples} \label{ejempl}

To use the previous regularization methods, we need to verify that the probability $P$ satisfies Eq.~\eqref{conditionker}, and some of the Hypotheses~\ref{hipbobe} or  \ref{hipunbobe}. In this section, we study some examples which satisfy the previous conditions. First, we discuss examples arising from the self-similar probabilities, that is, when $\rho(x)=e^x$ and $k=2$.
\par To check  Hypothesis \ref{hipunbobe}, it is sufficient to consider the case $j=0$. In fact, if Hypothesis \ref{hipunbobe} holds for $j=0$, then by triangle inequality we see that for all $j \neq 0$ and $\alpha \in [ 0,\min{(C, M/2 |j|)  } ]$
$$\frac{M}{2}\le M- |j|\alpha \le |k \foune(P)(\xi)-1+2\pi i \alpha \xi+ j \alpha|.$$
Thus, the Hypothesis \ref{hipunbobe} holds for all real numbers $j$. 
We now study for which values of $s$, and open sets $U$, the Hypothesis \ref{hipunbobe} holds. Without loss of generality, we assume that $j=0$.

\subsection{The equal-mitosis}
An important example of a self-similar case is the equal-mitosis. In equal-mitosis the conditional probability is given by $P= \delta_{x=\frac{1}{2}}$. In this case, the Fourier transform of $P$ is given by
$$ \foune(P)(\xi)=2 \left( \frac{1}{2} \right)^{\pi s-2 \pi i \xi}.$$
Thus, by the triangle inequality 
\begin{equation}
    \label{estimatieje}
    |2^{2-\pi s} -1| \le | \foune(P)(\xi) -1|.
\end{equation}
Therefore, we conclude that this probability  satisfies the Condition \eqref{conditionker} and Hypothesis \ref{hipbobe},  for all $s \in \R \setminus \{\frac{2}{\pi}\}$. 
In order to verify Hypothesis \ref{hipunbobe}, we consider the following cases.
\begin{itemize}
\item \textbf{First case:} $s >$ 2/$\pi$.
\newline For such $s$, the Hypothesis \ref{hipunbobe} is satisfied for all $\xi \in\R $. In fact, by  the triangle inequality
$$1-2^{2-\pi s} <\sqrt{1+(2\pi \alpha \xi)^2}-2^{2-\pi s}\le |2 \foune(P)(\xi)-1+2\pi i \alpha \xi|. $$
\item \textbf{Second case:} $s <$ 2/$\pi$.
\newline In this case, the Hypothesis \ref{hipunbobe} holds for each  open bounded set, and $\alpha$ small enough. In fact, using  triangle inequality we have
$$  |\foune(P)(\xi)-1 |-2\pi i \alpha \xi \le |2 \foune(P)(\xi)-1+2\pi i \alpha \xi|. $$
 Then, using estimate \ref{estimatieje}, we conclude that for $\alpha$ small the expression $|\foune(P)(\xi)-1 |-2\pi i \alpha \xi $ is positive.
\end{itemize}
Thus, we can apply our methodology to deal with the inverse problem for these cases. In Section  \ref{numer}, we will develop some numerical simulations for the previous cases and see the effectiveness of our approach.

%\subsection{The uniform re-partition} In this case  the probability is given by % $P=\mathbbm{1}_{[0,1]},$ and the Fourier transform is
%$$ \foune(P)(\xi)=\frac{1}{(-2i \xi+s)\pi}. $$
%Observe that using triangle inequality we have
%$$1- \frac {2}{ \pi s} \le|Re(2\foune(P)(\xi)-1)| \le |2\foune(P)(\xi)-1|. $$
%Thus, we conclude that condition $\eqref{conditionker}$ and Hypothesis %$\ref{hipbobe}$ hold for $s >$ 2/$\pi$. We also observe that
%$$1-\frac{2}{\pi s} <\sqrt{1+(2\pi \alpha \xi)^2}-2|\foune(P)(\xi)|\le |2 %\foune(P)(\xi)-1+2\pi i \alpha \xi|. $$
%Therefore,  Hypothesis \ref{hipunbobe} holds for $s >$ 2/$\pi$..

\section{Numerical solution of the inverse problem}\label{numer}

In this section, we recover numerically the solution of the inverse problem using the regularization methods proposed in Section \ref{secregu}. That is,  we recover $H$ from noisy measurements of $N$. If we assume that the noisy measurement $N_{\varepsilon}$ is a smooth function satisfying Condition \ref{suavidadcondi}, then there exists a unique solution $H^{\varepsilon}_{\alpha}$ for the Equation \eqref{meq} associated with $N_{\varepsilon}$. The purpose of this section is to explore numerically how close is $H^{\varepsilon}_{\alpha}$ to $H$, when the noisy data $N_{\varepsilon}$ is close to $N$. 
\newline To do that, we first construct an approximation for $N,$ using the numerical schemes proposed in \cite{dopezu,doti}, and then, we add random  noise  to the data $N.$
We construct the approximation $H^{\varepsilon}_{\alpha}$ using  Equation \eqref{apro}, where $\alpha$ is  the optimal regularization parameter for each method.

\subsection{Parametric specification.}

We now present some numerical simulations for  the equal-mitosis case. Here, the parameters are $s=0$, $k=2$ and $\rho(x)=e^x$, and the probability distribution is given by
$$ P=\delta_{\frac{1}{2}}.$$
To guarantee the existence and uniqueness of the solution of the direct problem, and also the condition Eq.~\eqref{condicre}, we select fragmentation and growth rates with fast decay at $0$ and $\infty$. See \cite{doga}. To be more specific, we use the growth rate $g(x)=xe^{-(x+1/x)}$ and the fragmentation rate $$B(x)=x^{2}e^{-(x+1/x)}.$$
In this experiment, we use the space $L^2_{\rho,0}$  to recover the function $H$ on $(0,10]$. We also use the parameter $j=1$ for the quasi-reversibility method and $m=10$ for the Landweber filter.
\subsection{Construction of \texorpdfstring{$N$.}{constructionn}}
We  construct the function $N$ on the interval $(0,3]$ using the numerical scheme proposed in \cite{dopezu,doti}. Here, the initial condition is given by
$$n_{0}(x)=e^{-(x-8)^{2}/2}.$$
We use  a regular grid on $ (0,3]$ with $500$ points. For the evolution process, we use a regular grid on $(0,200]$, with $10^{4}$ points.  We plot the function $N$ in Figure  \ref{constN}.

\begin{figure}
 \centering
  \subfloat[Equal-mitosis case]{
    \includegraphics[width=0.5\textwidth]{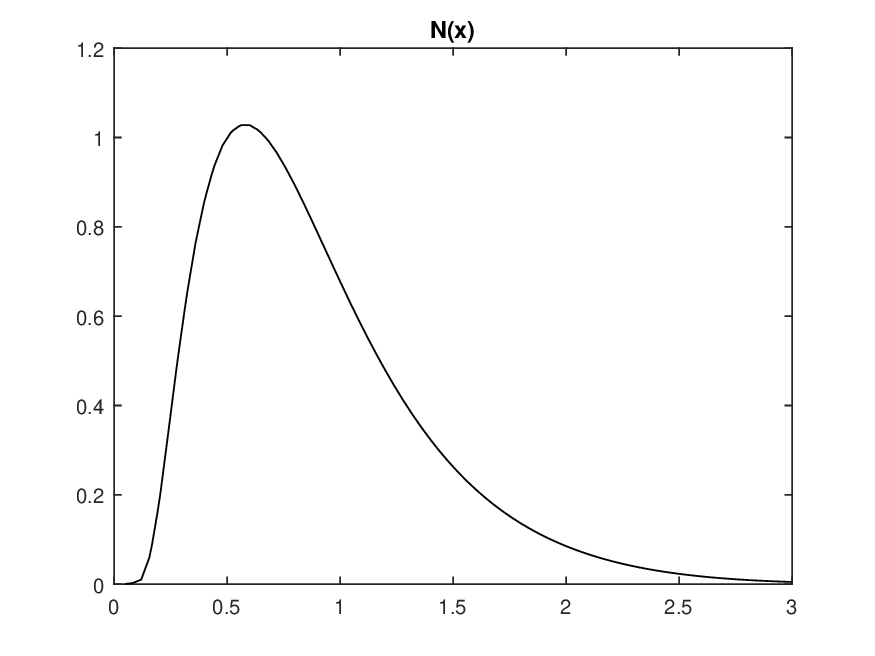}}
\caption{Construction of $N$, solution of the direct problem for the equal- mitosis case.}
\label{constN}
\end{figure}
\subsection{Reconstruction of \texorpdfstring{$H$ and $B$.}{handb}}
We now consider a noisy approximation $(N_{\varepsilon},\lambda_\varepsilon)$ for the eigenpair $(N,\lambda)$,  obtained by adding a random noise  to the data. That is, we assume that 
$$N_{\varepsilon}=\max (N+R_{\varepsilon},0), \quad \lambda_{\varepsilon}=\max (\lambda+S_{\varepsilon},0) ,$$
where $R_{\varepsilon},S_{\varepsilon}$ are random noises uniformly distributed in $[-\varepsilon, \varepsilon]$. We recover the approximation $H^{\varepsilon}_{\alpha}$ using the noisy measurement $(N_\varepsilon,\lambda_\varepsilon)$.  We plot the approximation $H^{\varepsilon}_{\alpha}$ for different noise levels $\varepsilon$. Here, we use the values $\varepsilon=10^{-2}$ (Figure \ref{reconstr1}), and $\varepsilon= 10^{-3}$ (Figure \ref{reconstr2}). The parameters used are $m=10$ for the Tikhonov method and $j=1$ for the Landweber method.
\par To recover the fragmentation rate $B$ from $H=BN$ we use the truncate division by $N$. That is, we define $B(x)= H(x)/N(x)$ if $N(x)\neq 0,$ and zero otherwise. The following figures show the recovered function $B$ for the parameters $\varepsilon=10^{-2}$, and  $\varepsilon= 10^{-3}$. Observe that the instabilities near  $x=3$ of the reconstructed function $B$ are due to the fact that the  fast decay of $N$ near $x=3$ affects the  truncated division.

\begin{figure}
 \centering
  \subfloat[Reconstruction of H]{
    \includegraphics[width=0.5\textwidth]{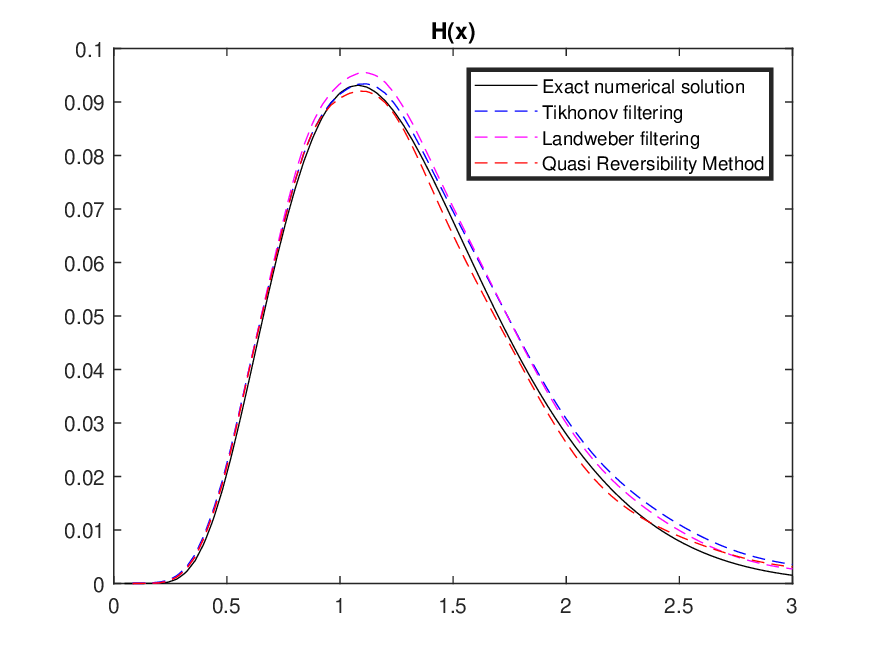}}
  \subfloat[Reconstruction of B]{
    \includegraphics[width=0.5\textwidth]{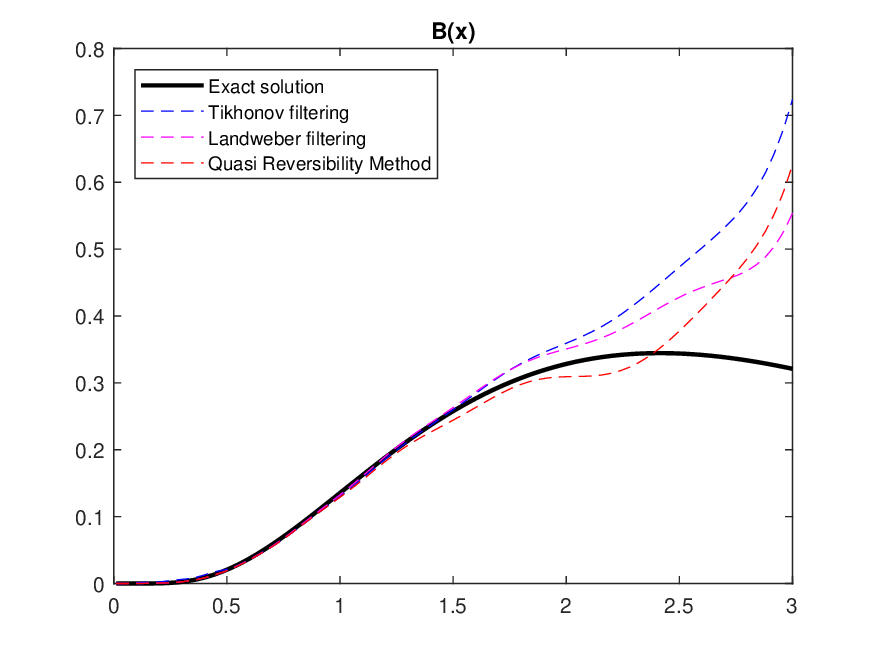}}
 \caption{Numerical reconstruction of $H$ and $B$, using the noise level  $\varepsilon=10^{-2}$.}
 \label{reconstr1}
\end{figure}
\begin{figure}
 \centering
  \subfloat[Reconstruction of H]{
    \includegraphics[width=0.5\textwidth]{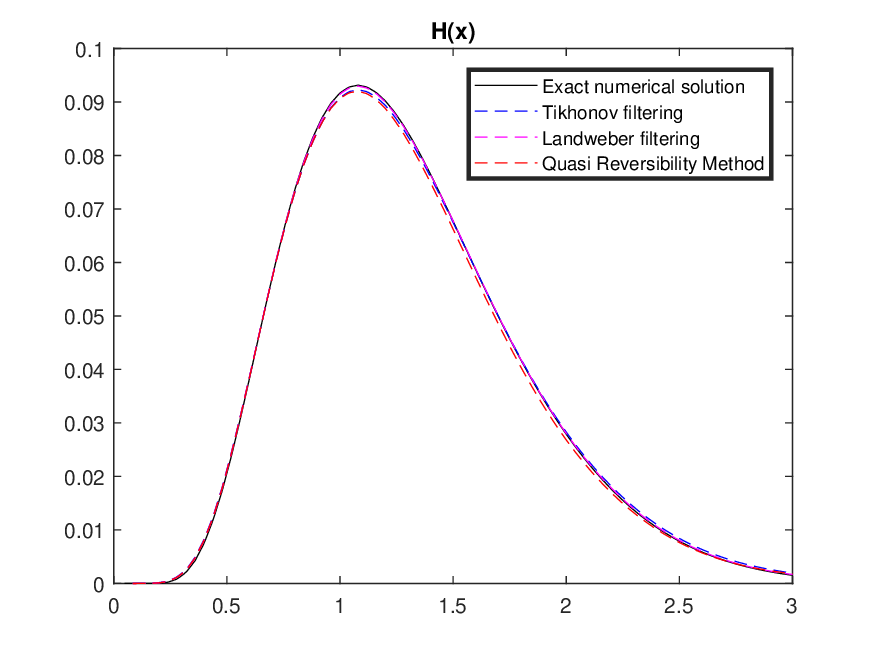}}
  \subfloat[Reconstruction of B]{
    \includegraphics[width=0.5\textwidth]{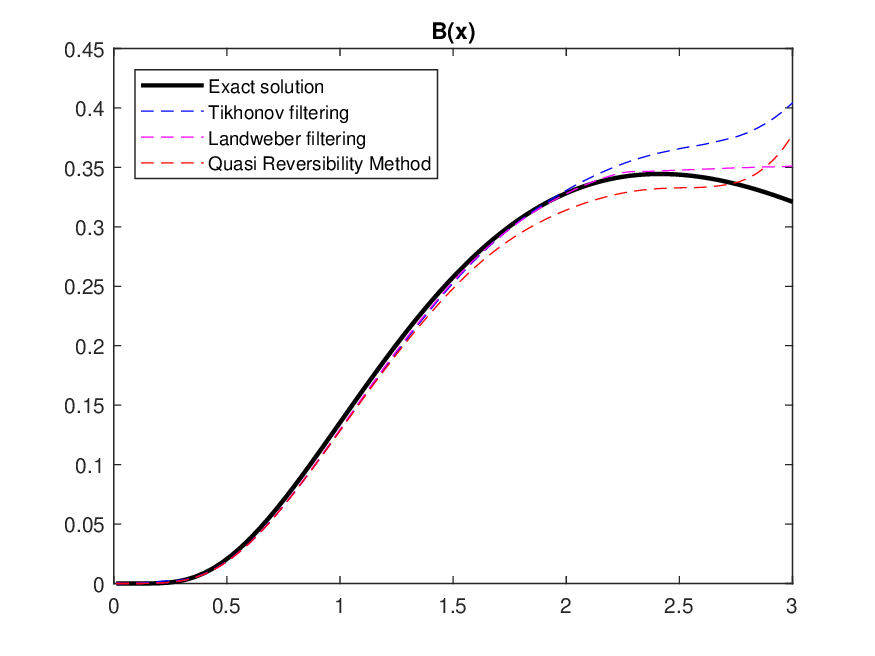}}
 \caption{Numerical reconstruction of $H$ and $B$, using the noise level  $\varepsilon=10^{-3}$.}
 \label{reconstr2}
\end{figure}

\subsection{Numerical Error.}

We compare the numerical error of the reconstructions of the functions $B$ and $H$ for small values of $ \varepsilon $. Here, we use the $L^2 [0,3]$ norm to estimate this error. This norm is computed using rectangular integration. For this experiment, we assume that $\epsilon \in [10^{-4},0.5]$. We plot the error in logarithmic scale in Figure \ref{figuerro}. Observe that for small values of $\epsilon$ the Landweber filter gives a better reconstruction, which is following our theoretical results.
\begin{figure}
 \centering
  \subfloat[Numerical error of the reconstruction of $B$ ]{
    \includegraphics[width=0.5\textwidth]{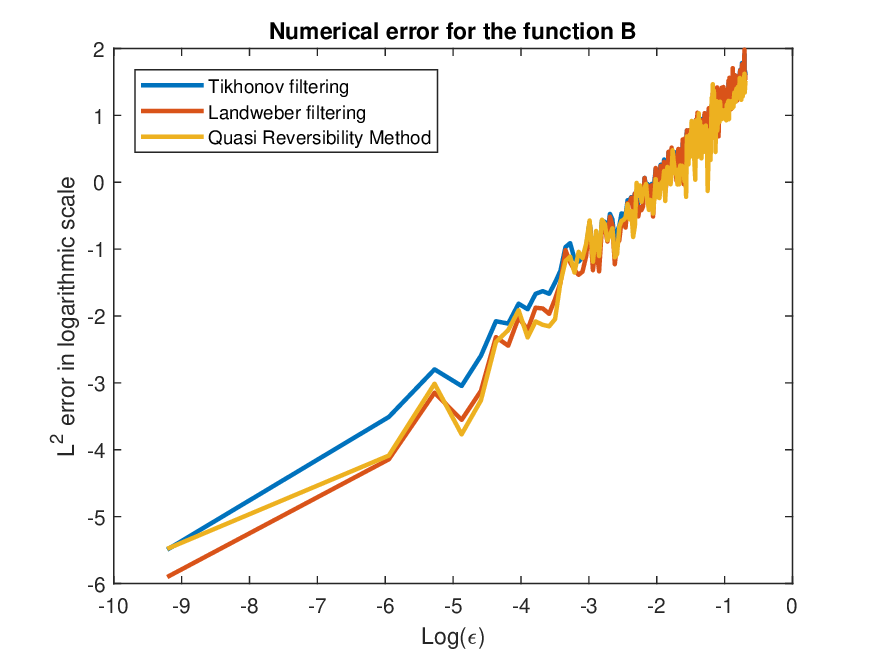}}
  \subfloat[Numerical error of the reconstruction of $H$]{
    \includegraphics[width=0.5\textwidth]{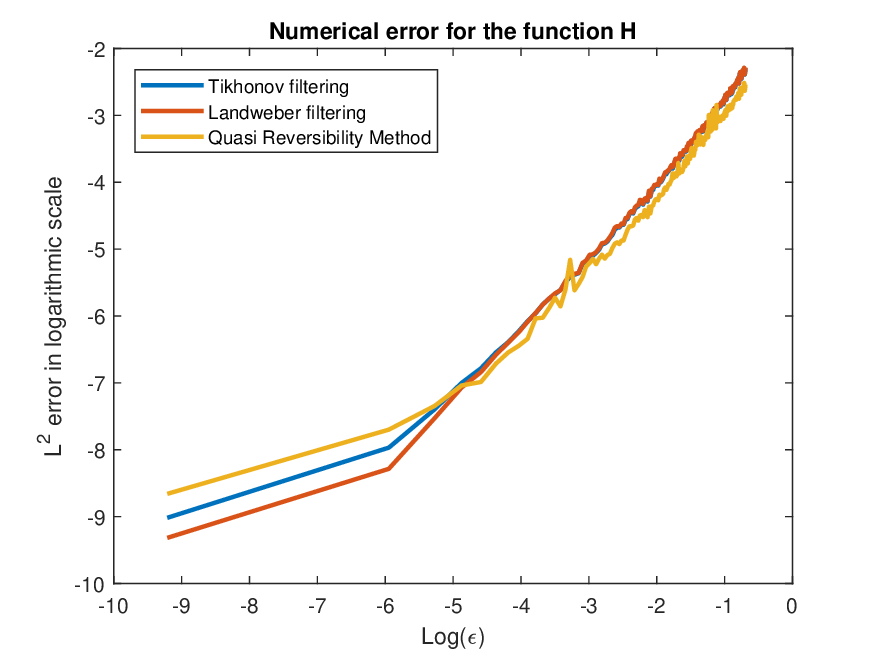}}
 \caption{Numerical error of the reconstruction of the functions $B$ and $H$ using the $L^2$ norm for several values of $\epsilon$. Note that errors are on a logarithmic scale}
 \label{figuerro}
\end{figure}

\section{Conclusion}

In this article,  we treated the inverse problem for the growth-fragmentation equation of a wide class of transition probabilities. Namely, we deal with transport probabilities which are generalizations of self-similar probabilities.
We developed a new approach to regularize the inverse problem associated with the transport probabilities. Our approach is based on the Fourier transform theory for locally compact groups. We regularized the Fourier transform of differential operators using several filters in the spectral variable, such as modifications of the  Landweber and Tikhonov filters.
\par For each method, we obtained their respective error estimates. The Landweber method provides an algorithm to recover the fragmentation rate $B$, with order  $O(\varepsilon^\frac{2m}{2m+1})$, where $m$ is the degree of smoothness of $B$, as proved in Theorem \ref{landw}. The error obtained using the Landweber method is better compared with other methods. This fact was verified by numerical simulations.
\par Our theoretical approach has been focused on transport operation induced by diffeomorphisms. A natural continuation of this work is to deal with transport operations induced by arbitrary functions.  Another possible follow-up is to apply our methodology to problems arising in data science, as well as biological problems as in \cite{conclus1,conclu2}.

One potential application of the problem under consideration concerns modeling  normal prion protein (PrP(C)) and infectious prion protein (PrP(Sc)) populations interacting in an infected host \cite{propol1}. Indeed, Perthame and Doumic in \cite{calvodoumicperthame} proved key asymptotic results for the stable population where our results could be applied.

\par 
% We note that even in the particular zero-growth case,  $g\equiv 0$,  we have an interesting topic for further exploration since it allows for some simplifications in the operator of Equation~\eqref{meq}. In such cases, Doumic et al. \cite{DET} have addressed the issue of reconstructing the conditional probability kernel under a self-similar assumption and knowledge of the fragmentation rate $B$. To our knowledge, the reconstruction of the more general conditional probability kernels that we discuss in the current article has not been analyzed yet. This shall be addressed in a future publication.
We note that even in the particular zero-growth case,  $g\equiv0$,  we have an interesting topic for further exploration since it allows for some simplifications in the operator of Equation~\eqref{meq}. In such cases Doumic {\it et al.} \cite{DET} have addressed the issue of reconstructing the conditional probability kernel under a self-similar assumption and knowledge of the fragmentation rate $B$. This shall be addressed in a future publication.

\section*{Acknowledgements}
AAG and JPZ acknowledge support from the FSU-2020-09 grant from Khalifa University. The authors acknowledge the financial support provided by CAPES,  Coordenação de Aperfeiçoamento de Pessoal de Nível Superior, grant number 88887.311757/2018-00, CNPq, Conselho Nacional de Desenvolvimento Científico e Tecnológico, grant numbers 308958/2019-5 and 307873/2013-7, and FAPERJ, Fundação Carlos Chagas Filho de Amparo à Pesquisa do Estado do Rio de Janeiro, grant numbers E-26/200.899/2021 and  E-26/202.927/2017.

\section*{References}

\bibliography{bibli}

\end{document}